\documentclass[11pt]{amsart}
\usepackage[margin=1in]{geometry}

\usepackage{amssymb}
\usepackage{amsthm}
\usepackage{amsmath}
\usepackage{mathrsfs}
\usepackage{amsbsy}
\usepackage[all]{xy}
\usepackage{bm}
\usepackage{hyperref}
\usepackage{tikz}
\usepackage{array}
\usepackage{float}
\usepackage{enumerate}
\usepackage{xcolor}
\usepackage{hhline}
\allowdisplaybreaks
\usepackage[noadjust]{cite}

\usepackage{caption}
\usepackage{subcaption}
\usepackage{tabu}
\usepackage{diagbox}
\usepackage{tikz}
\usepackage{bbm}
\usepackage{booktabs}

\DeclareFontFamily{U}{rcjhbltx}{}
\DeclareFontShape{U}{rcjhbltx}{m}{n}{<->rcjhbltx}{}
\DeclareSymbolFont{hebrewletters}{U}{rcjhbltx}{m}{n}

\let\aleph\relax\let\beth\relax

\DeclareMathSymbol{\aleph}{\mathord}{hebrewletters}{39}
\DeclareMathSymbol{\beth}{\mathord}{hebrewletters}{98}

\usepackage[noabbrev,capitalise]{cleveref}

\newenvironment{enumerate*}%
  {\begin{enumerate}[(I)]%
    \setlength{\itemsep}{10pt}%
    \setlength{\parskip}{0pt}}%
  {\end{enumerate}}

\newtheorem{theorem}{Theorem}[section]
\newtheorem{proposition}[theorem]{Proposition}
\newtheorem{corollary}[theorem]{Corollary}

\newtheorem{claim}[theorem]{Claim}

\theoremstyle{definition}

\title[Inequalities among higher-order difference sets]{Inequalities among higher-order difference sets, or,\\ remarks on a construction of Ruzsa}

\author{Noah Kravitz}
\address{St John's College, Oxford and Mathematical Institute, University of Oxford; St Giles', Oxford OX1 3JP, UK}
\email{noah.kravitz@maths.ox.ac.uk}

\begin{document}


\begin{abstract}
Extending a little-known construction of Ruzsa, we characterize the quadruples $(s,t,u,v)$ of nonnegative integers such that the inequality $|sA-tA| \leq |uA-vA|$ holds for all finite sets $A \subseteq \mathbb{Z}$.
\end{abstract}

\maketitle

\section{Introduction}
\subsection{Higher-order difference sets}
For subsets $A,B$ of an abelian group, define the sumset
$$A+B:=\{a+b: a\in A, ~b \in B\}.$$
For a nonnegative integer $r$, we denote the $r$-fold iterated sumset of $A$ by
$$rA:=\underbrace{A+\cdots+A}_r=\{a_1+\cdots+a_r: a_1, \ldots, a_r \in A\}$$
(with the convention $0A:=\{0\}$ for $A$ nonempty).

Inequalities among cardinalities of iterated sumsets play a central role in additive combinatorics.  One prominent example is the Pl\"unnecke--Ruzsa inequality (see \cite{petridis}). 
At the same time, there is interest in constructing pathological examples to show that certain natural-looking sumset inequalities cannot hold in general.  For instance, this perspective has motivated the now-abundant literature (see \cite{km} and the references therein) around more-sums-than-differences (MSTD) sets, namely, finite sets $A \subseteq \mathbb{Z}$ with $|2A|>|A-A|$.

In this paper, we study the higher-order difference set inequality
\begin{equation}\label{eq:main}
|sA-tA| \leq |uA-vA|,
\end{equation}
where $s,t,u,v$ are fixed nonnegative integers.  One source of motivation is that \eqref{eq:main} captures a vast generalization of the MSTD setup.  It is clear that \eqref{eq:main} always holds (in every imaginable setting) if $s \leq u$ and $t \leq v$, or if $s\leq v$ and $t \leq u$.  Our first main result shows that this sufficient condition is in fact necessary, in the setting where $A$ ranges over the finite subsets of the integers.

\begin{theorem}\label{thm:main}
Let $s,t,u,v$ be nonnegative integers.  Then the inequality $$|sA-tA| \leq |uA-vA|$$ holds for all finite sets $A \subseteq \mathbb{Z}$ if and only if $\min(s,t) \leq \min(u,v)$ and $\max(s,t) \leq \max(u,v)$.
\end{theorem}

The hard direction of the proof boils down to producing examples of finite sets $A \subseteq \mathbb{Z}$ satisfying \emph{a priori} surprising inequalities such as
\begin{equation}\label{eq:examples}
|10A-10A|>|9A-100000A| \quad \text{and} \quad |100A|>|99A-99A|.
\end{equation}
The existence of such examples is striking because one generally expects iterated sumsets to grow with the number of summands.  For comparison, the MSTD set problem concerns the relative sizes of $2A$ and $A-A$, which are both $2$-fold sumsets, whereas in each example in \eqref{eq:examples}, the sumset on the left-hand side has much lower order than the sumset on the right-hand side.  In the second example, the set $50A$ must be an MSTD set, in fact in a rather strong sense since $50A+50A=100A$ is not only larger than $50A-50A$ but also larger than $99A-99A=(50A-50A)+(49A-49A)$.

We remark that once the inequality \eqref{eq:main} fails for some finite set $A \subseteq \mathbb{Z}$, it necessarily fails by a wide margin for other $A$'s, since a single counterexample can be amplified as follows.  Suppose that $B \subseteq \mathbb{Z}$ satisfies $|sB-tB|>|uB-vB|$.  Then set $\alpha:=\log(|sB-tB|)/\log(|uB-vB|)>1$, and notice that the $n$-fold Cartesian product of $B$ satisfies
$$|sB^n-tB^n|=|sB-tB|^n=|uB-vB|^{\alpha n}=|uB^n-vB^n|^{\alpha}.$$
Taking $A:=B^n$ with $n$ large (and applying a suitable Freiman homomorphism to transfer from $\mathbb{Z}^n$ to $\mathbb{Z}$), we see that \eqref{eq:main} even fails polynomially with an exponent $1+\Omega_{s,t,u,v}(1)$.

This observation might make one wonder if \eqref{eq:main} can fail only for ``trivial reasons'' due to numerical coincidences for small sets $A$.  To the contrary, our proof of Theorem~\ref{thm:main} will, without any amplification, naturally produce examples with $|sA-tA|/|uA-vA|$ arbitrarily large.

\subsection{Sums of dilates}
The arguments behind Theorem~\ref{thm:main} naturally output a more general statement about inequalities among sums of dilates.  For a nonzero integer $\lambda$, define the dilate
$$\lambda \cdot A:=\{\lambda a: a \in A\}$$
(not to be confused with the $\lambda$-fold sumset $\lambda A$).  A natural generalization of the problem resolved by Theorem~\ref{thm:main} is determining the pairs of multisets of nonzero integers $\{\lambda_1, \ldots, \lambda_\ell\}, \{\mu_1, \ldots, \mu_m\}$ such that the inequality
\begin{equation}\label{eq:dilates}
|\lambda_1 \cdot A+\cdots+\lambda_\ell \cdot A| \leq |\mu_1 \cdot A+\cdots+\mu_m \cdot A|
\end{equation}
holds for all finite sets $A \subseteq \mathbb{Z}$.  It is clear that \eqref{eq:dilates} always holds if some scalar multiple of $\{\mu_1, \ldots, \mu_m\}$ contains $\{\lambda_1, \ldots, \lambda_\ell\}$ (as multisets).  This is, however, not the most general way in which \eqref{eq:dilates} can trivially hold: Say that $\{\mu_1, \ldots, \mu_m\}$ \emph{dominates} $\{\lambda_1, \ldots, \lambda_\ell\}$ if there are a nonzero real number $r$ and disjoint subsets $S_1, \ldots, S_\ell \subseteq [m]$ such that $$r\lambda_j=\sum_{h \in S_j} \mu_h$$ for all $j \in [\ell]$.
For example $\{2,3,5\}$ dominates $\{2,8\}$ (with $r=1$) and $\{-1,-4\}$ (with $r=-2$).  If $\{\mu_1, \ldots, \mu_m\}$ dominates $\{\lambda_1, \ldots, \lambda_\ell\}$ in this way, then $\mu_1 \cdot A+\cdots+\mu_m \cdot A$ contains a translate of
$$\sum_{h \in S_1} (\mu_h \cdot A)+\cdots+\sum_{h \in S_\ell} (\mu_h \cdot A) \supseteq r\lambda_1 \cdot A+\cdots+r\lambda_\ell \cdot A=r \cdot (\lambda_1 \cdot A+\cdots+\lambda_\ell \cdot A),$$
whence $|\lambda_1 \cdot A+\cdots+\lambda_\ell \cdot A| \leq |\mu_1 \cdot A+\cdots+\mu_m \cdot A|$.  Our most general result says that this is the only way for \eqref{eq:dilates} to hold.

\begin{theorem}\label{thm:dilates}
Let $\lambda_1, \ldots, \lambda_\ell, \mu_1, \ldots, \mu_m$ be nonzero integers.  Then the inequality
$$|\lambda_1 \cdot A+\cdots+\lambda_\ell \cdot A| \leq |\mu_1 \cdot A+\cdots+\mu_m \cdot A|$$
holds for all finite sets $A \subseteq \mathbb{Z}$ if and only if $\{\mu_1, \ldots, \mu_m\}$ dominates $\{\lambda_1, \ldots, \lambda_\ell\}$.
\end{theorem}

Theorem~\ref{thm:main} corresponds to the special case of Theorem~\ref{thm:dilates} where $\lambda_i, \mu_i \in \{-1,1\}$ for all $i$.

\subsection{Relation to prior work}

The first hint of the possibility of achieving \eqref{eq:examples} came from an unpublished (and somewhat obscure) 2016 note of Ruzsa~\cite{ruzsa} in which he constructed sets with many more differences than multiple sums.  More precisely, he showed that for every positive integer $u$, there is a finite set $A \subseteq \mathbb{Z}$ with
$$|A-A|>|uA|.$$
This corresponds to the special case of Theorem~\ref{thm:main} where $s=t=1$ and $v=0$.  Our proof of Theorem~\ref{thm:main} is based on a careful analysis and generalization of Ruzsa's quite creative construction.  Our main contribution is recognizing the appropriate level of generality at which his approach works.  Ruzsa says that he took inspiration from an older paper of Haight~\cite{hai}.  See \cite{mil, nat} for other subsequent work based on Ruzsa's construction.

One other special case of Theorem~\ref{thm:main} was previously known: The ``homogeneous'' instance $s+t=u+v$ recovers a result of Iyer, Lazarev, Miller, and Zhang~\cite{ilmz} on what they term ``generalized MSTD sets''.\footnote{See \cite{km} for corrections of several mistakes in \cite{ilmz}.}  This work used fairly ad hoc methods which do not play a role in the present paper.

We remark that the $(s,t,u,v)=(2,0,1,1)$ instance of Theorem~\ref{thm:main} provides a new source of MSTD sets.  Our construction is more ``robust'' than many existing constructions in the literature, in the sense that it naturally provides amplification-free examples with $2A$ much larger than $|A-A|$.

For other work on sums of dilates with a similar general flavor, see the references in \cite{pon,hp}.

\subsection{Open problems}
We round out this introduction with a few open problems.

\subsubsection{Growth rates}
As we suggested above, the natural scaling for the problem of comparing $|sA-tA|$ and $|uA-vA|$ is captured by the quantity
$$\alpha(s,t,u,v):=\lim_{n \to \infty} \sup_{|A|=n} \frac{\log|sA-tA|}{\log |uA-vA|}.$$ The limit exists by Fekete's Lemma (since sumset sizes are multiplicative with respect to the Cartesian product), and it trivially lies between $1$ and $s+t$.  In the 1970's, Freiman and Pigaev~\cite{fp} studied the values of the constants $\alpha(2,0,1,1)$ and $\alpha(1,1,2,0)$, which pertain to the relative sizes of $|2A|$ and $|A-A|$. See \cite{ruzsa2,hry} for follow-up work.  Here we raise the more general problem of computing or estimating $\alpha(s,t,u,v)$ for various quadruples $(s,t,u,v)$.  Notice that Theorem~\ref{thm:main} characterizes the quadruples with $\alpha(s,t,u,v)=1$.

\subsubsection{Arbitrary dilates}
Theorem~\ref{thm:dilates} is stated for integer dilates of finite sets of integers.  The same result of course holds for rational dilates, since the statement is invariant under scaling all of the dilates by a fixed nonzero real.  One can further ask about arbitrary complex  dilates of finite sets of complex numbers.  This setting is not covered by our proof of Theorem~\ref{thm:dilates}, and we leave the corresponding characterization as an open problem. 

\subsubsection{Three-legged sumset races}
Nathanson~\cite{nat2} recently introduced a family of problems about what have come to be termed ``sumset races''.  The most basic version asks for a pair of finite sets $A,B \subseteq \mathbb{Z}$ such that $|hA|-|hB|$ has prescribed sign for all $1 \leq h \leq H$.  Such pairs of sets do always exist, and indeed somewhat more can be achieved (see~\cite{kra,fkz,pr}).  Iyer, Lazarev, Miller, and Zhang~\cite{ilmz} have studied a related problem about higher-order difference sets.  They showed that for any fixed positive integer $H$, there is a finite set $A \subseteq \mathbb{Z}$ such that the quantities $$|HA|, ~|(H-1)A-A|, ~\ldots, ~|\lceil H/2 \rceil A-\lfloor H/2 \rfloor A|$$ have prescribed relative order.

We propose extending these problems to mixed sumsets of varying orders, in what might  be described as ``three-legged sumset races''.  For instance, suppose that $(s_1, t_1), \ldots, (s_{H^2},t_{H^2})$ is an ordering of $[0,H-1]^2$ in which $(s,t)$ always occurs before $(s+1,t)$ and $(s,t+1)$.  Must there exist finite sets $A,B \subseteq \mathbb{Z}$ such that
$$|s_1A+t_1B|<|s_2A+t_2B|<\cdots<|s_{H^2}A+t_{H^2}B|?$$
For a version with difference sets, consider the ordered pairs $(s,t)$ with $0 \leq t \leq s \leq H-1$, and suppose that $(s_1, t_1), \ldots, (s_{H(H+1)/2},t_{H(H+1)/2})$ is an ordering of these pairs in which $(s,t)$ always occurs before $(s+1,t)$ and (if $s>t$) before $(s,t+1)$.
Must there exist a finite set $A \subseteq \mathbb{Z}$ such that
$$|s_1A-t_1A|<|s_2A-t_2A|<\cdots<|s_{H(H+1)/2}A-t_{H(H+1)/2}A|?$$
Theorem~\ref{thm:main} tells us that no single one of these inequalities precludes the existence of such a set $A$.  It would be quite interesting to understand if new obstructions can arise from combinations of several inequalities.

\subsection{Organization of the paper}
In Section~\ref{sec:mstd} we prove Theorem~\ref{thm:main} in the special case corresponding to MSTD sets.  We highlight this case both because it is of particular interest and because it illustrates the main ideas in a simpler (and less notation-intensive) setting.  In Section~\ref{sec:dilates} we prove Theorem~\ref{thm:dilates} in full generality and show how to deduce Theorem~\ref{thm:main} from it.

\section{A robust MSTD construction} \label{sec:mstd}

The goal of this section is to prove the following result about MSTD sets in finite cyclic groups.

\begin{proposition}\label{prop:mstd}
For every $\varepsilon>0$, there are $q \in \mathbb{N}$ and $A \subseteq \mathbb{Z}/q\mathbb{Z}$ such that $$2A=\mathbb{Z}/q\mathbb{Z} \quad \text{and} \quad |A-A| \leq \varepsilon q.$$
\end{proposition}

Lifting this construction from $\mathbb{Z}/q\mathbb{Z}$ to the integers is routine.

\begin{corollary}\label{cor:mstd}
For every $K>0$, there is a finite set $A \subseteq \mathbb{Z}$ such that
$$|2A|\geq K|A-A|.$$
\end{corollary}

\begin{proof}
Proposition~\ref{prop:mstd} with $\varepsilon:=1/(2K)$ produces a set $B \subseteq \mathbb{Z}/q\mathbb{Z}$ (for some $q \in \mathbb{N}$) such that
$$2B=\mathbb{Z}/q\mathbb{Z} \quad \text{and} \quad |B-B| \leq q/(2K).$$
Let $A \subseteq [1,q]$ be the preimage of $B$ under the canonical projection map.  Then
$$|2A| \geq |2B|=q \quad \text{and} \quad |A-A| \leq 2|B-B| \leq q/K,$$
which gives the desired inequality $|2A| \geq K|A-A|$.
\end{proof}

To motivate the proof strategy of Proposition~\ref{prop:mstd}, notice that $2A$ contains the element $x \in \mathbb{Z}/q\mathbb{Z}$ if and only if there is some $\varphi \in \mathbb{Z}/q\mathbb{Z}$ such that $\varphi, x-\varphi \in A$.  It follows that $2A=\mathbb{Z}/q\mathbb{Z}$ if and only if there is some function $\varphi: \mathbb{Z}/q\mathbb{Z} \to \mathbb{Z}/q\mathbb{Z}$ such that $\varphi(x), x-\varphi(x) \in A$ for all $x \in \mathbb{Z}/q\mathbb{Z}$.  With this in mind, we will take
$$A:=\bigcup_{x \in \mathbb{Z}/q\mathbb{Z}} \{\varphi(x), x-\varphi(x)\}$$
for a carefully chosen function $\varphi: \mathbb{Z}/q\mathbb{Z} \to \mathbb{Z}/q\mathbb{Z}$ which forces $A-A$ to lie in a particular small subset of $\mathbb{Z}/q\mathbb{Z}$.  We will achieve this by individually analyzing the various ``types'' of differences in $A-A$, such as $\varphi(x)-(x-\varphi(x))$ and $\varphi(x)-\varphi(y)$ and $\varphi(x)-(y-\varphi(y))$.

\begin{proof}[Proof of Proposition~\ref{prop:mstd}]
Let $p_0=p_0(\varepsilon)$ be an odd prime to be determined later.  Let $T=T(\varepsilon)$ be a positive integer, depending only on $p_0$, to be determined later; let $p_1=p_1(\varepsilon), \ldots, p_T=p_T(\varepsilon)$ be distinct odd primes (also distinct from $p_0$), depending on $T,\varepsilon$, to be determined later.  Set $q:=p_0p_1 \cdots p_T$, and view the group $\mathbb{Z}/q\mathbb{Z}$ as
$$\mathbb{Z}/q\mathbb{Z}\cong \mathbb{Z}/p_0\mathbb{Z} \times \cdots \times \mathbb{Z}/p_T\mathbb{Z};$$
we accordingly coordinatize elements of $\mathbb{Z}/q\mathbb{Z}$ as $x=(x_0, \ldots, x_T)$.  Our subset $A \subseteq \mathbb{Z}/q\mathbb{Z}$ will satisfy
\begin{equation}\label{eq:diff-constrained}
A-A \subseteq \bigcup_{0 \leq j \leq T}\left( \prod_{t=0}^{j-1} (\mathbb{Z}/p_t\mathbb{Z}) \times \{0\} \times \prod_{t=j+1}^T (\mathbb{Z}/p_t\mathbb{Z}) \right),
\end{equation}
whence by a union bound
$$|A-A| \leq q(p_0^{-1}+p_1^{-1}+\cdots+p_T^{-1}).$$
The latter quantity  will be at most $\varepsilon q$ as long as $p_0$ is sufficiently large (say, at least $2\varepsilon^{-1}$) and $p_1, \ldots, p_T$ are sufficiently large (say, at least $2T\varepsilon^{-1}$).  Note here that once we are given $\varepsilon$, we first choose $p_0$, then specify $T$ depending on $p_0$, and only at the end choose $p_1, \ldots, p_T$.

We will take
$$A:=\bigcup_{x \in \mathbb{Z}/q\mathbb{Z}} \{\varphi(x), x-\varphi(x)\}$$
for a carefully chosen function $\varphi: \mathbb{Z}/q\mathbb{Z} \to \mathbb{Z}/q\mathbb{Z}$.  From $$\varphi(x)+(x-\varphi(x))=x$$ it is immediate that $2A=\mathbb{Z}/q\mathbb{Z}$.  The main work is arranging for \eqref{eq:diff-constrained} to hold.  Let us write $\varphi$ in coordinates as
$$\varphi(x)=(\varphi_0(x), \ldots, \varphi_T(x)).$$
Notice that $A-A$ consists of the elements of the form
\begin{equation}\label{eq:types-of-diff}
\begin{aligned}
E_1(x,y) &:= \varphi(x)-\varphi(y),
&\qquad
E_2(x,y) &:= (x-\varphi(x))-\varphi(y), \\
E_3(x,y) &:= \varphi(x)-(y-\varphi(y)),
&\qquad
E_4(x,y) &:= (x-\varphi(x))-(y-\varphi(y))
\end{aligned}
\end{equation}
for $x,y \in \mathbb{Z}/q\mathbb{Z}$, since we can subtract either of $\varphi(y), y-\varphi(y)$ from either of $\varphi(x), x-\varphi(x)$.

We start by specifying
$$\varphi_0(x_0. \ldots, x_T):=x_0/2$$
(with the division by $2$ taking place in $\mathbb{Z}/p_0 \mathbb{Z}$).  This guarantees that all of the expressions in \eqref{eq:types-of-diff} have vanishing $\mathbb{Z}/p_0 \mathbb{Z}$ coordinate whenever $x,y$ have the same $\mathbb{Z}/p_0 \mathbb{Z}$ coordinate.  It remains to consider the expressions in \eqref{eq:types-of-diff} for $x,y$ with $x_0 \neq y_0$.

Take $T:=4p_0(p_0-1)$ (the particular value is not very important), and let $$(a(1),b(1),i(1)), \quad \ldots, \quad (a(T),b(T),i(T))$$
be a listing of the triples $(a,b,i)$ where $a,b \in \mathbb{Z}/p_0\mathbb{Z}$ are distinct and $i \in [4]$.
For each $t \in [T]$, we define the corresponding coordinate function $\varphi_{t}$ as follows to guarantee that $E_{i(t)}(x,y)$ has $\mathbb{Z}/p_t\mathbb{Z}$ coordinate $0$ whenever $(x_0,y_0)=(a(t),b(t))$.
\begin{itemize}
    \item If $i=1$, then set
$$\varphi_t(z):=0.$$
\item If $i=2$, then set
$$\varphi_t(z):=\begin{cases}
z_t, &\text{if } z_0=a;\\
0, &\text{otherwise}.
\end{cases}$$
\item If $i=3$, then set
$$\varphi_t(z):=\begin{cases}
z_t, &\text{if } z_0=b;\\
0, &\text{otherwise}.
\end{cases}$$
\item If $i=4$, then set
$$\varphi_t(z):=z_t.$$
\end{itemize}
Thus, whenever $x_0,y_0$ are distinct and $i \in [4]$, there is some index $t \in [T]$ such that the $t$-th coordinate of $E_i(x,y)$ vanishes.  This, together with the previous paragraph, establishes \eqref{eq:diff-constrained} and completes the proof.  (We remark that in each of our four cases, the value of $\varphi_t(z)$ matters only when $z_0 \in \{a,b\}$, and it can be set arbitrarily otherwise.)
\end{proof}

\section{The general setting}\label{sec:dilates}

The goal of this section is to prove the following proposition, which immediately implies the hard direction of Theorem~\ref{thm:dilates} (as in the deduction of Corollary~\ref{cor:mstd} in the previous section).

\begin{proposition}\label{prop:dilates}
Let $\lambda_1, \ldots, \lambda_\ell, \mu_1, \ldots, \mu_m$ be nonzero integers such that $\{\mu_1, \ldots, \mu_m\}$ does not dominate $\{\lambda_1, \ldots, \lambda_\ell\}$.  Then for every $\varepsilon>0$, there are $q \in \mathbb{N}$ and $A \subseteq \mathbb{Z}/q\mathbb{Z}$ such that
$$\lambda_1 \cdot A+\cdots+\lambda_\ell \cdot A=\mathbb{Z}/q\mathbb{Z} \quad \text{and} \quad |\mu_1 \cdot A+\cdots+\mu_m \cdot A| \leq \varepsilon q.$$
\end{proposition}

We will start with a discussion of the proof of Proposition~\ref{prop:mstd} and an indication of how the argument can be generalized to the setting of Proposition~\ref{prop:dilates}.  We will then provide a detailed proof of Proposition~\ref{prop:dilates}; this is necessarily rather notation-heavy, and we hope that it is elucidated by the preceding informal discussion.  Finally, we will deduce Theorems~\ref{thm:dilates} and~\ref{thm:main}.

\subsection{Overall strategy}

The crux of the proof of Proposition~\ref{prop:mstd} was our ability to force vanishing coordinates for the quantities $E_i(x,y)$.  The expressions $E_1$ and $E_4$ are easier to handle because they are symmetric in $x$ and $y$: For instance, naively setting $\varphi_t(z):=z_t$ for all $z$ guarantees that $$E_4(x,y)=(x-\varphi(x))-(y-\varphi(y))$$ has vanishing $t$-th coordinate.  The expressions $E_2$ and $E_3$ are more challenging because they treat $x,y$ asymmetrically.  For $$E_2(x,y)=(x-\varphi(x))-\varphi(y),$$ for example, one might wish to set $\varphi_t(x):=x_t$ and $\varphi_t(y):=0$ in order to ensure that both summands in $E_2(x,y)$ have vanishing $t$-th coordinate.  The trouble is that we don't know in advance whether a given element $z \in \mathbb{Z}/q\mathbb{Z}$ will play the role of $x$ or the role of $y$ here, and we cannot set $\varphi_t(z)$ to be $z_t$ and $0$ simultaneously.

The solution in the proof is to use the $\mathbb{Z}/p_0\mathbb{Z}$ coordinate as an indexing set that witnesses this asymmetry: If we are given $(x_0,y_0)$ in advance and $x_0 \neq y_0$, then we can assign the roles of $x,y$ in $E_2(x,y)$ on the basis of the $\mathbb{Z}/p_0\mathbb{Z}$ coordinate.  More precisely, to ensure that both $x_t-\varphi_t(x)$ and $\varphi_t(y)$ vanish, we can define $\varphi_t(z)$ to be  $z_t$ whenever $z_0=x_0$, and to be $0$ whenever $z_0=y_0$; the point here is that we were essentially able to define $\varphi$ independently on $x$ and on $y$ because of our prior information about $x_0,y_0$.  The final piece of the puzzle is noticing that in the remaining case where $x_0=y_0$, we can guarantee the vanishing of the $\mathbb{Z}/p_0\mathbb{Z}$ coordinate because $$(x_0-\varphi_0(x))-\varphi_0(y)=x_0-2\varphi_0(x)$$ depends on only a single variable (and hence there are no complications from asymmetry).  One can view this argument as an inductive procedure with respect to the number of distinct variables appearing, which Ruzsa in \cite{ruzsa} calls the ``level'' of an expression.  In the setting of Proposition~\ref{prop:mstd}, all of our expressions have level at most $2$ because $A-A$ is a $2$-fold sumset.  In the general setting of Proposition~\ref{prop:dilates}, our expressions can have level up to $m$, so our induction has $m$ steps rather than $2$ steps.

The first move in the proof of Proposition~\ref{prop:dilates} is setting
\[
A := \bigcup_{x \in \mathbb{Z}/q\mathbb{Z}}
\left\{
\lambda_1^{-1}\varphi^{(1)}(x),~
\ldots,~
\lambda_{\ell}^{-1}\varphi^{(\ell)}(x)
\right\},
\]
where $\varphi^{(1)}, \ldots, \varphi^{(\ell-1)}: \mathbb{Z}/q\mathbb{Z} \to \mathbb{Z}/q\mathbb{Z}$ are functions to be specified later and we have defined
\begin{equation}\label{eq:phi-ell}
\varphi^{(\ell)}(x):=x - \varphi^{(1)}(x) - \cdots - \varphi^{(\ell-1)}(x).
\end{equation}
By design we have
\begin{align*}
x &= \lambda_1\bigl(\lambda_1^{-1}\varphi^{(1)}(x)\bigr)
  + \cdots
  + \lambda_{\ell}\bigl(\lambda_{\ell}^{-1}\varphi^{(\ell)}(x)\bigr)\in \lambda_1 \cdot A+\cdots+\lambda_\ell \cdot A
\end{align*}
for all $x \in \mathbb{Z}/q\mathbb{Z}$, whence $\lambda_1 \cdot A+\cdots+\lambda_\ell \cdot A=\mathbb{Z}/q\mathbb{Z}$.  The introduction of so many auxiliary functions $\varphi^{(i)}$ initially looks intimidating, but the ability to specify them independently will supply us with useful flexibility when we constrain $\mu_1 \cdot A+\cdots+\mu_m \cdot A$.

The elements of $\mu_1 \cdot A+\cdots+\mu_m \cdot A$ come in finitely many ``types'', analogous to the expressions $E_1, E_2, E_3, E_4$ from the proof of Proposition~\ref{prop:mstd}.  Now, our types are the expressions
$$E=E_{i_1, \ldots, i_m}(x^{(j_1)}, \ldots, x^{(j_m)}):=\mu_1 \lambda_{i_1}^{-1} \varphi^{(i_1)}(x^{(j_1)})+\cdots+\mu_m\lambda_{i_m}^{-1} \varphi^{(i_m)}(x^{(j_m)}),$$
for various choices of $i_1, \ldots, i_m \in [\ell]$ and $j_1, \ldots, j_m \in [m]$ (the variables $x^{(j)}$ correspond to $x,y$ from the proof of Proposition~\ref{prop:mstd}).

Recalling and substituting the definition~\eqref{eq:phi-ell} of $\varphi^{(\ell)}$, let us view each $E$ as a formal rational linear combination of the variables $$\varphi^{(1)}(x^{(j)}), ~\ldots, ~\varphi^{(\ell-1)}(x^{(j)}), ~x^{(j)}$$
for $j \in [m]$.  Our inductive indexing procedure will let us act as if the functions $\varphi^{(1)}, \ldots, \varphi^{(\ell-1)}$ are chosen independently for the various $j$'s.  For each $j \in [m]$, we consider the part of $E$ involving $x^{(j)}$ (either bare or in $\varphi^{(i)}(x^{(j)})$ for some $i \in [\ell-1]$). 
It is here that the non-domination hypothesis comes into play: The key point (see Claim~\ref{claim:local-cancel} below) is that the $x^{(j)}$ part of $E$ can never be merely a nonzero scalar multiple of $x^{(j)}$.  In other words, if the $x^{(j)}$ coefficient is nonzero, then there is some $i_* \in [\ell-1]$ such that the $\varphi^{(i_*)}(x^{(j)})$ coefficient is also nonzero.  Thus, by setting $\varphi^{(i_*)}_t(z)$ to be an appropriate scalar multiple of $z_t$ and setting $\varphi^{(i)}_t(z):=0$ for all $i \neq i_*$, we can guarantee that the $t$-th coordinate of $E$ vanishes.  

\subsection{Proof of Proposition~\ref{prop:dilates}}
We now carry out the strategy from the previous subsection in full detail.  The reader may wish to keep in mind the simpler setting of Proposition~\ref{prop:mstd}.  The heart of the construction is Claim~\ref{claim:local-cancel}, and the inductive scheme is captured in Claim~\ref{claim:vanishing}.

\begin{proof}[Proof of Proposition~\ref{prop:dilates}]
There is nothing to show if $m=0$, so we may assume that $m \geq 1$.  Since $\{\mu_1\}$ dominates every $1$-element set, we may also assume that $\ell \geq 2$.  In the proof, we will encounter various rationals depending only on $\lambda_1, \ldots, \lambda_\ell, \mu_1, \ldots, \mu_m$.  We will assume that all primes $p$ appearing in the proof are sufficiently large that multiplication by these rationals makes sense in $\mathbb{Z}/p\mathbb{Z}$.

We will make use of finite index sets $T_1, \ldots, T_{m}$ and primes $p_t$ for $t \in T_k$.  The primes will all be distinct and large.  We will first specify $T_1$ and choose primes $\{p_t\}_{t \in T_1}$ such that $\sum_{t \in T_1}p_t^{-1} \leq \varepsilon/m$.  Next, for each $2 \leq k \leq m$, we will inductively specify $T_k$ depending on $\{p_t\}_{T_{k'}}$ for $k'<k$ and then choose primes $\{p_t\}_{t \in T_k}$ such that $\sum_{t \in T_k}p_t^{-1} \leq \varepsilon/m$.  Finally, we will set
$$q:=\prod_{k=1}^m \prod_{t \in T_k}p_t.$$
Since all of the $p_t$'s are distinct, the Chinese Remainder Theorem gives
$$\mathbb{Z}/q\mathbb{Z} \cong \prod_{k=1}^m \prod_{t \in T_k}(\mathbb{Z}/p_t\mathbb{Z}).$$
We will arrange for $\mu_1 \cdot A+\cdots+\mu_m \cdot A$ to be contained in the subset of $\mathbb{Z}/q\mathbb{Z}$ consisting of the elements with at least one vanishing coordinate.  This will yield the desired estimate
\begin{equation}\label{eq:coord-planes}
|\mu_1 \cdot A+\cdots+\mu_m \cdot A| \leq q \sum_{k=1}^m \sum_{t \in T_k}p_t^{-1} \leq q\sum_{k=1}^m \varepsilon/m=\varepsilon q
\end{equation}
due to our choice of the $p_t$'s.

With this setup out of the way, we embark on the proof proper.  As in the proof sketch in the previous subsection, let $\varphi^{(1)}, \ldots, \varphi^{(\ell-1)}: \mathbb{Z}/q\mathbb{Z} \to \mathbb{Z}/q\mathbb{Z}$ be functions to be specified later, and set
$$\varphi^{(\ell)}(x):=x - \varphi^{(1)}(x) - \cdots - \varphi^{(\ell-1)}(x).$$
We define
\[
A := \bigcup_{x \in \mathbb{Z}/q\mathbb{Z}}
\left\{
\lambda_1^{-1}\varphi^{(1)}(x),~
\ldots,~
\lambda_{\ell}^{-1}\varphi^{(\ell)}(x)
\right\},
\]
which guarantees that $\lambda_1 \cdot A+\cdots+\lambda_\ell \cdot A=\mathbb{Z}/q\mathbb{Z}$.  It remains to show that every element of $\mu_1 \cdot A+\cdots+\mu_m \cdot A$ has at least one vanishing coordinate.  Let us coordinatize each $\varphi^{(i)}$ as
\[
\varphi^{(i)}(x) := \bigl(\varphi^{(i)}_{t}\bigr)_{t \in T_k,\, k \in [m]}.
\]

Now, $\mu_1 \cdot A+\cdots+\mu_m \cdot A$ consists of the elements of the form
$$E_{i_1, \ldots, i_m}(x^{(1)}, \ldots, x^{(m)}):=\mu_1\lambda_{i_1}^{-1} \varphi^{(i_1)}(x^{(1)})+\cdots+\mu_m\lambda_{i_m}^{-1} \varphi^{(i_m)}(x^{(m)}),$$
for various choices of $i_1, \ldots, i_m \in [\ell]$, where $x^{(1)}, \ldots, x^{(m)} \in \mathbb{Z}/q\mathbb{Z}$.  At this point the elements $x^{(1)}, \ldots, x^{(m)}$ may coincide.  Let us write $x_{\leq k}$ for the $\prod_{k'=1}^k \prod_{t \in T_{k'}}(\mathbb{Z}/p_t\mathbb{Z})$ component of an element $x \in \mathbb{Z}/q\mathbb{Z}$.  We further discriminate among the above expressions according to the number of distinct values among $$x^{(1)}_{\leq k}, \ldots, x^{(m)}_{\leq k}$$
for various choices of $1 \leq k \leq m$; this corresponds to the ``level'' of Ruzsa that we mentioned previously.  The precise statement that we require goes as follows.

\begin{claim}[inductive vanishing]\label{claim:vanishing}
There are finite index sets $T_1, \ldots, T_m$, corresponding choices of primes $p_t$, and functions $\varphi^{(1)}, \ldots, \varphi^{(\ell-1)}$ such that the following holds for all $1 \leq k \leq m$.  If $x^{(1)}, \ldots, x^{(m)} \in \mathbb{Z}/q\mathbb{Z}$ are such that there are at most $k$ distinct values among $x^{(1)}_{\leq k}, \ldots, x^{(m)}_{\leq k}$, then for each choice of $i_1, \ldots, i_m \in [\ell]$, the element $$E_{i_1, \ldots, i_m}(x^{(1)}, \ldots, x^{(m)}) \in \mathbb{Z}/q\mathbb{Z}$$ has some vanishing coordinate among the coordinates indexed by $T_1, \ldots, T_k$.
\end{claim}

The proposition is a quick consequence of the $k=m$ instance of this claim, since for each choice of $x^{(1)}, \ldots, x^{(m)} \in \mathbb{Z}/q\mathbb{Z}$, there are trivially at most $m$ distinct values among $x^{(1)}_{\leq m}, \ldots, x^{(m)}_{\leq m}$. 
The claim then ensures that the element $E_{i_1, \ldots, i_m}(x^{(1)}, \ldots, x^{(m)}) \in \mathbb{Z}/q\mathbb{Z}$ has at least one vanishing coordinate for each choice of $i_1, \ldots, i_m \in [\ell]$, and the proposition follows from \eqref{eq:coord-planes}.  Before we prove the claim, we isolate the step where we use the non-domination hypothesis to produce a ``building block'' for our functions $\varphi^{(i)}$.

\begin{claim}[building block]\label{claim:local-cancel}
Let $P \subseteq [m]$, and let $i_h \in [\ell]$ for each $h \in P$.  Then there are linear functions $\psi^{(1)}, \ldots, \psi^{(\ell-1)}: \mathbb{Z} \to \mathbb{Q}$ such that, on setting
\begin{equation}\label{eq:psi-ell}
\psi^{(\ell)}(z):=z-\psi^{(1)}(z)-\cdots-\psi^{(\ell-1)}(z),
\end{equation}
we have the identity
\begin{equation}\label{eq:building-block-identity}
\sum_{h \in P} \mu_h \lambda_{i_h}^{-1} \psi^{(i_h)}(z)=0.
\end{equation}
\end{claim}

\begin{proof}[Proof of Claim~\ref{claim:local-cancel}]
For each $i \in [\ell]$, set $$S_i:=\{h \in P: i_h=i\},$$
so that (expanding the definition of $\varphi^{(\ell)}$) we can express
\begin{align*}
\sum_{h \in P} \mu_h \lambda_{i_h}^{-1} \psi^{(i_h)}(z) &=\sum_{i=1}^\ell \sum_{h \in S_i} \mu_h \lambda_i^{-1} \psi^{(i)}(z)\\
 &=\sum_{i=1}^{\ell-1} \left(\sum_{h \in S_i} \mu_h \lambda_i^{-1} -\sum_{h \in S_\ell} \mu_h \lambda_\ell^{-1} \right) \psi^{(i)}(z)+\sum_{h \in S_\ell} \mu_h \lambda_\ell^{-1}z.
\end{align*}
Define $r:=\sum_{h \in S_\ell} \mu_h \lambda_\ell^{-1}$.  If $r=0$, then we can take $\psi^{(i)}(z):=0$ for all $i \in [\ell-1]$.  Now suppose that $r \neq 0$.  The non-domination hypothesis tells us that there is some $i_* \in [\ell]$ such that
$$\sum_{h \in S_{i_*}}\mu_{h} \neq r\lambda_{i_*},$$
and the definition of $r$ guarantees that $i_* \neq \ell$.  The coefficient of $\psi^{(i_*)}(z)$ in the above expression is
$$s:=\sum_{h \in S_{i_*}} \mu_h \lambda_{i_*}^{-1} -\sum_{h \in S_\ell} \mu_h \lambda_\ell^{-1}=\sum_{h \in S_{i_*}} \mu_h \lambda_{i_*}^{-1} -r,$$
which is nonzero.  We now take
$$\psi^{(i)}(z):=\begin{cases}
-(r/s) z, &\text{if } i=i_*;\\
0, &\text{if }i \in [\ell-1] \setminus \{i_*\}.
\end{cases}$$
This yields \eqref{eq:building-block-identity} because the $\psi^{(i_*)}(z)$ terms cancel the $z$ terms, and all other terms vanish.
\end{proof}

We now prove Claim~\ref{claim:vanishing} by induction on $k$.  At the $k$-th step, we specify the data of $$T_{k}, ~\{p_t\}_{t \in T_{k}}, ~ \{\varphi^{(1)}_t, \ldots, \varphi^{(\ell-1)}_t\}_{t \in T_{k}}$$
such that the conclusion of the claim holds for the relevant value of $k$.

\begin{proof}[Proof of Claim~\ref{claim:vanishing}]

{\bf  Base case ($k=1$).}
Let $T_1$ consist of the tuples $$t=(i_1, \ldots, i_m) \in [\ell]^m,$$
and take $\{p_t\}_{t \in T_1}$ to be sufficiently large primes, as in the second paragraph of the proof of the proposition.  Fix some $t \in T_1$.  We will specify coordinate functions $$\varphi^{(1)}_{t}, \ldots, \varphi^{(\ell-1)}_{t}: \mathbb{Z}/q\mathbb{Z}\to \mathbb{Z}/p_{t}\mathbb{Z}$$ with the property that
$E_{i_1, \ldots, i_m}(x^{(1)}, \ldots, x^{(m)})$
has vanishing $\mathbb{Z}/p_{t}\mathbb{Z}$ coordinate whenever $x^{(1)}_{\leq 1}, \ldots, x^{(m)}_{\leq 1}$ are all equal (in fact, whenever $x^{(1)}_{t}, \ldots, x^{(m)}_{t}$ are all equal).

Claim~\ref{claim:local-cancel} with $P:=[m]$ provides functions $$\psi^{(1)}, \ldots, \psi^{(\ell)}: \mathbb{Z}/p_t\mathbb{Z} \to \mathbb{Z}/p_t\mathbb{Z}$$ (the rational coefficients defining these functions can be interpreted in $\mathbb{Z}/p_t\mathbb{Z}$ because we assumed that $p_t$ is sufficiently large) satisfying \eqref{eq:psi-ell} and \eqref{eq:building-block-identity}.  We now set
$$\varphi^{(i)}_t(x):=\psi^{(i)}(x_t)$$
for all $i \in [\ell-1]$ and note that the same formula extends to $i=\ell$.  Finally, we check that whenever $x^{(1)}, \ldots, x^{(m)} \in \mathbb{Z}/q\mathbb{Z}$ satisfy $x^{(1)}_t=\cdots=x^{(m)}_t=z$, the $\mathbb{Z}/p_{t}\mathbb{Z}$ coordinate of $E_{i_1, \ldots, i_m}(x^{(1)}, \ldots, x^{(m)})$ is precisely the quantity on the left-hand side of \eqref{eq:building-block-identity}, which vanishes.  This completes the base case.

{\bf Induction step ($2 \leq k \leq m$).} Suppose that we have already specified the relevant data for $k' \leq k-1$.  We now wish to specify $$T_{k}, ~\{p_t\}_{t \in T_{k}}, ~ \{\varphi^{(1)}_t, \ldots, \varphi^{(\ell-1)}_t\}_{t \in T_{k}}$$
in order to handle choices of $x^{(1)}, \ldots, x^{(m)} \in \mathbb{Z}/q\mathbb{Z}$ with at most $k$ distinct values among $x^{(1)}_{\leq k}, \ldots, x^{(m)}_{\leq k}$.  If there are at most $k-1$ distinct values among $x^{(1)}_{\leq k-1}, \ldots, x^{(m)}_{\leq k-1}$, then by induction $$E_{i_1, \ldots, i_m}(x^{(1)}, \ldots, x^{(m)})$$ already has some vanishing coordinate among the coordinates indexed by $T_1, \ldots, T_{k-1}$.  Thus it suffices to consider the case where there are exactly $k$ distinct values among $x^{(1)}_{\leq k-1}, \ldots, x^{(m)}_{\leq k-1}$ and exactly $k$ distinct values among $x^{(1)}_{\leq k}, \ldots, x^{(m)}_{\leq k}$.  In particular, the coordinates indexed by $T_1, \ldots, T_{k-1}$ record which of the $x^{(j)}_{\leq k}$'s coincide, in the sense that
\begin{equation*}\label{eq:k-level-coincide}
x^{(j)}_{\leq k}=x^{(j')}_{\leq k} \quad \text{if and only if} \quad x^{(j)}_{\leq k-1}=x^{(j')}_{\leq k-1}.
\end{equation*}
This pattern of coincidences induces a partition of $[m]$ into $k$ parts, and our plan is to apply Claim~\ref{claim:local-cancel} to each part individually.

We introduce the index set $T_k$ whose elements encode all of the possible patterns of coincidences.  More precisely, let $T_k$ consist of the tuples
$$t=(\mathcal{P}; a^{(1)}, \ldots, a^{(k)}; i_1, \ldots, i_m),$$
where
\begin{itemize}
    \item $\mathcal{P}=\{P_1, \ldots, P_k\}$ is a partition of $[m]$ into $k$ parts;\footnote{Here we let $\mathcal{P}$ range over the partitions of $[m]$ into $k$ \emph{unordered} parts, and for each $\mathcal{P}$ we list the parts $P_1, \ldots, P_k$ in a fixed order, e.g., lexicographically.}
    \item $a^{(1)}, \ldots, a^{(k)}$ are distinct elements of $\prod_{k'=1}^{k-1} \prod_{t \in T_{k'}}(\mathbb{Z}/p_t\mathbb{Z})$;
    \item $(i_1, \ldots, i_m) \in [\ell]^m$.
\end{itemize}
Notice that the finite set $T_k$ depends only on the parameters $\ell,m$ and on the previously specified primes $\{p_t\}_{t \in T_{k'}}$ for $k' \leq k-1$.  Take $\{p_t\}_{t \in T_k}$ to be sufficiently large primes, as in the second paragraph of the proof of the proposition.  Fix some $t \in T_k$.  We will specify coordinate functions $$\varphi^{(1)}_{t}, \ldots, \varphi^{(\ell-1)}_{t}: \mathbb{Z}/q\mathbb{Z}\to \mathbb{Z}/p_t\mathbb{Z}$$ with the property that
$E_{i_1, \ldots, i_m}(x^{(1)}, \ldots, x^{(m)})$
has vanishing $\mathbb{Z}/p_{t}\mathbb{Z}$ coordinate whenever
\begin{equation}\label{eq:condition-on-lifts}
x^{(h)}_{\leq k-1}=a^{(j)} \quad \text{for all } h \in P_j, \quad \text{and} \quad x^{(h)}_{\leq k}=x^{(h')}_{\leq k} \quad \text{for all } h,h' \in P_j
\end{equation}
(in fact from the latter condition we will use only that $x^{(h)}_{t}=x^{(h')}_{t}$). 

For each $j \in [k]$, Claim~\ref{claim:local-cancel} with $P:=P_j$ provides functions
$$\psi^{(j,1)}, \ldots, \psi^{(j,\ell)}: \mathbb{Z}/p_t\mathbb{Z} \to \mathbb{Z}/p_t\mathbb{Z}$$
satisfying \eqref{eq:psi-ell} and \eqref{eq:building-block-identity}.  We now set
$$\varphi_t^{(i)}(x):=\psi^{(j,i)}(x_t) \quad \text{whenever } x_{\leq k-1}=a^{(j)}$$
for all $i \in [\ell-1]$, and the same formula extends to $i=\ell$.  We finish specifying the functions $\varphi^{(i)}_t$ by setting
$$\varphi_t^{(i)}:=0 \quad \text{whenever } x_{\leq k-1} \notin \{a^{(1)}, \ldots, a^{(k)}\}.$$
Let us check that these functions satisfy the desired property.  Suppose that $x^{(1)}, \ldots, x^{(m)} \in \mathbb{Z}/q\mathbb{Z}$ satisfy \eqref{eq:condition-on-lifts}.  Then there are $z^{(1)}, \ldots, z^{(k)} \in \mathbb{Z}/p_t\mathbb{Z}$ such that $x^{(h)}_t=z^{(j)}$ whenever $h \in P_j$.  Now the $\mathbb{Z}/p_{t}\mathbb{Z}$ coordinate of $E_{i_1, \ldots, i_m}(x^{(1)}, \ldots, x^{(m)})$ is
\begin{align*}
\mu_1\lambda_{i_1}^{-1} \varphi_t^{(i_1)}(x^{(1)})+\cdots+\mu_m\lambda_{i_m}^{-1} \varphi_t^{(i_m)}(x^{(m)}) &=\sum_{j \in [k]} \sum_{h \in P_j} \mu_h \lambda_{i_h}^{-1} \psi^{(j,i_h)}(z^{(j)}).
\end{align*}
Each inner sum over $h \in P_j$ vanishes by construction due to \eqref{eq:building-block-identity}, so the entire expression also vanishes, as desired.  This completes the induction step.
\end{proof}

With Claim~\ref{claim:vanishing} established, we have completed the proof of the proposition.
\end{proof}

We remark that in the proof of Claim~\ref{claim:vanishing}, we could have started the induction at $k=0$ (where the statement is vacuously true) and interpreted the $k=1$ argument as the first instance of the induction step.  We chose to spell out the $k=1$ argument as a warm-up for the more notation-heavy material to follow.


\subsection{Deduction of the main theorems}
We are now equipped to deduce Theorems~\ref{thm:dilates} and~\ref{thm:main}.

\begin{proof}[Proof of Theorem~\ref{thm:dilates}]
We explained in the introduction that if $\{\mu_1, \ldots, \mu_m\}$ dominates $\{\lambda_1, \ldots, \lambda_\ell\}$, then the inequality
$$|\lambda_1 \cdot A+\cdots+\lambda_\ell \cdot A| \leq |\mu_1 \cdot A+\cdots+\mu_m \cdot A|$$
holds for all finite sets $A \subseteq \mathbb{Z}$.  Suppose now that $\{\mu_1, \ldots, \mu_m\}$ does not dominate $\{\lambda_1, \ldots, \lambda_\ell\}$.  Then for any $\varepsilon>0$, Proposition~\ref{prop:dilates} produces a set $B \subseteq \mathbb{Z}/q\mathbb{Z}$ (for some $q \in \mathbb{N}$) such that
$$\lambda_1 \cdot B+\cdots+\lambda_\ell \cdot B=\mathbb{Z}/q\mathbb{Z} \quad \text{and} \quad |\mu_1 \cdot B+\cdots+\mu_m \cdot B| \leq \varepsilon q.$$
Let $A \subseteq [1,q]$ be the preimage of $B$ under the canonical projection map.  Then
$$|\lambda_1 \cdot A+\cdots+\lambda_\ell \cdot A| \geq |\lambda_1 \cdot B+\cdots+\lambda_\ell \cdot B|=q$$ and $$|\mu_1 \cdot A+\cdots+\mu_m \cdot A| \leq (|\mu_1|+\cdots+|\mu_m|) \cdot |\mu_1 \cdot B+\cdots+\mu_m \cdot B| \leq \varepsilon(|\mu_1|+\cdots+|\mu_m|)q.$$
The latter is smaller than $q$ if $\varepsilon$ is chosen sufficiently small in terms of $\mu_1, \ldots, \mu_m$.
\end{proof}

As mentioned in the introduction, Theorem~\ref{thm:main} is the special case of dilates $\pm 1$.

\begin{proof}[Proof of Theorem~\ref{thm:main}]
Let $s,t,u,v$ be nonnegative integers.  The theorem is immediate if $(s,t)=(0,0)$ or $(u,v)=(0,0)$, so assume that $s+t, u+v>0$.  Let $\{\lambda_1, \ldots, \lambda_\ell\}$ consist of $s$ copies of $1$ and $t$ copies of $-1$, and let $\{\mu_1, \ldots, \mu_m\}$ consist of $u$ copies of $1$ and $v$ copies of $-1$.  Notice that if $\{\mu_1, \ldots, \mu_m\}$ dominates $\{\lambda_1, \ldots, \lambda_\ell\}$ with some scaling $r$, then the domination also occurs with scaling $r/|r| \in \{-1,1\}$.  Thus $\{\mu_1, \ldots, \mu_m\}$ dominates $\{\lambda_1, \ldots, \lambda_\ell\}$ if and only if either $s \leq u$ and $t \leq v$ (corresponding to the scaling $r=1$), or $s\leq v$ and $t \leq u$ (corresponding to the scaling $r=-1$).  The result now follows from Theorem~\ref{thm:dilates}.
\end{proof}

\section*{Acknowledgments and AI usage}

The author was supported in part by a NSF Mathematical Sciences Postdoctoral Research Fellowship under grant DMS-2501336.  We used ChatGPT for literature search and proofreading.

\end{document}